\documentclass[12pt]{article}
\usepackage{amssymb,amsmath,graphicx,amsfonts}
\usepackage{amsmath}
\usepackage{amsfonts}
\usepackage{subcaption}
\usepackage{tikz}
\usetikzlibrary{arrows}
\usepackage{color}

\newtheorem{theorem}{Theorem}[section]

\newtheorem{lemma}[theorem]{Lemma}
\newtheorem{proposition}[theorem]{Proposition}
\newtheorem{corollary}[theorem]{Corollary}

\newtheorem{definition}[theorem]{Definition}
\newtheorem{conjecture}[theorem]{Conjecture}
\newtheorem{thm}{Theorem}
\newtheorem{problem}[thm]{Problem}
\newtheorem{preproof}{{\bf Proof}}

\newenvironment{proof}[1]{\begin{preproof}{\rm
               #1}\hfill{$\blacksquare$}}{\end{preproof}}
\newtheorem{presproof}{{\bf Sketch of Proof.\ }}

\newtheorem{prepro}{{\bf Proposition}}

\title{Simultaneous coloring of vertices and incidences of Outerplanar graphs}
{\small
\author{Mahsa Mozafari-Nia$^a$, Moharram N. Iradmusa$^{a,b}$\\
{\small $^{a}$Department of Mathematical Sciences, Shahid Beheshti University,}\\
{\small G.C., P.O. Box 19839-63113, Tehran, Iran.}\\
{\small $^{b}$School of Mathematics, Institute for Research in Fundamental Sciences (IPM),}\\
{\small P.O. Box: 19395-5746, Tehran, Iran.}}
\begin{document}
\maketitle
\begin{abstract}
A $vi$-simultaneous proper $k$-coloring of a graph $G$ is a coloring of all vertices and incidences of the graph in which any two adjacent or incident elements in the set $V(G)\cup I(G)$ receive distinct colors, where $I(G)$ is the set of incidences of $G$. The $vi$-simultaneous chromatic number, denoted by $\chi_{vi}(G)$, is the smallest integer $k$ such that $G$ has a $vi$-simultaneous proper $k$-coloring.	In [M. Mozafari-Nia, M. N. Iradmusa, A note on coloring of $\frac{3}{3}$-power of subquartic graphs, Vol. 79, No.3, 2021]  $vi$-simultaneous proper coloring of graphs with maximum degree $4$ is investigated and they conjectured that for any graph $G$ with maximum degree $\Delta\geq 2$, $vi$-simultaneous proper coloring of $G$ is at most $2\Delta+1$.
In [M. Mozafari-Nia, M. N. Iradmusa, Simultaneous coloring of vertices and incidences of graphs, arXiv:2205.07189, 2022]  the correctness of the conjecture for some classes of graphs such as  $k$-degenerated graphs, cycles, forests, complete graphs, regular bipartite graphs is investigated. In this paper, we prove that the $vi$-simultaneous chromatic number of any outerplanar graph $G$ is either $\Delta+2$ or $\Delta+3$, where $\Delta$ is the maximum degree of $G$.
\end{abstract}
\section{Introduction}\label{sec1}
All graphs we consider in this paper are simple, finite and undirected. For a graph $G$, we denote its vertex set, edge set, face set (if $G$ is planar), maximum degree
and maximum size of cliques of $G$  by $V(G)$, $E(G)$, $F(G)$, $\Delta(G)$ and $\omega(G)$, respectively.  Also, for a vertex $v\in V(G)$, $N_G(v)$ is the set of neighbors of $v$ in $G$ and any vertex of degree $k$ is called a $k$-vertex. From now on,  we use the notation $[n]$ instead of $\{1,\ldots,n\}$. We mention some of the definitions that are referred to throughout the note and for other necessary definitions and notations, we refer the reader to a standard text-book \cite{bondy}.\\
For any graph $G$, apart from vertices and edges of the graph as its elements, we have incidences of $G$ as other elements of the graph. The concepts of incidences of a graph were introduced by Brualdi and Massey in 1993 \cite{Bruldy}. In graph $G$, any pair $i=(v,e)$ is called an incidence of $G$, if $v\in V(G)$, $e\in E(G)$ and $v\in e$. Also, in this case the elements $v$ and $i$ are called incident. The set of the incidences of $G$ is denoted by $I(G)$.
\begin{definition} {\em{\cite{Bruldy}}}\label{inci}
Let $G=(V,E)$ be a multigraph. The incidence graph of $G$, denoted by $\mathcal{I}(G)$, defined with vertex set $I(G)$ and with an edge between any pair of incidences $(v,e)$ and $(w,f)$ provided one of the following holds:
	\begin{enumerate}
		\item[(1)]  $v=w$,
		\item[(2)] $e=f$,
		\item[(3)] the edge $\{v,w\}$ equals $e$ or $f$.
	\end{enumerate}
\end{definition}
For any edge $e=\{u,v\}$, we call $(u,e)$, the first incidence of $u$ and $(v,e)$, the second incidence of $u$. In general, for a vertex $v\in V(G)$, the set of the first incidences and the second incidences of $v$ is denoted by $I_1^G(v)$ and $I_2^G(v)$, respectively. Also, let $I_G(v)=I_1^G(v)\cup I_2^G(v)$, $I_G[v]=\{v\}\cup I_G(v)$, $I_1^G[v]=\{v\}\cup I_1^G(v)$ and $I_2^G[v]=\{v\}\cup I_2^G(v)$. Sometime we remove the index $G$ for simplicity.\\
A mapping $c$ from $V(G)$ to $[k]$ is a proper $k$-coloring of $G$, if $c(v)\neq c(u)$ for any two adjacent vertices $u$ and $v$. A minimum integer $k$ that $G$ has a proper $k$-coloring is the chromatic number of $G$ and denoted by $\chi(G)$.\\
Instead of the vertices, we can color the edges or the incidences of a graph. A mapping $c$ from $E(G)$ to $[k]$ is a proper edge-$k$-coloring of $G$, if $c(e)\neq c(e')$ for any two adjacent edges $e$ and $e'$ ($e\cap e'\neq\varnothing$). A minimum integer $k$ that $G$ has a proper edge-$k$-coloring is the chromatic index of $G$ and denoted by $\chi'(G)$. Similarly, any proper $k$-coloring of $\mathcal{I}(G)$ is an incidence $k$-coloring of $G$. The incidence chromatic number of $G$, denoted by $\chi_i(G)$, is the minimum integer $k$ such that $G$ is incidence $k$-colorable.\\
Another coloring of a graph is simultaneous coloring in which we color two or three kinds of elements of the graph at the same time subject to some constraints. The first and the most well-known simultaneous coloring of graphs is total coloring which was introduced by Behzad in $1965$ \cite{behzad}. A mapping $c$ from $V(G)\cup E(G)$ to $[k]$ is a proper total $k$-coloring of $G$, if $c(x)\neq c(y)$ for any two adjacent or incident elements $x$ and $y$. A minimum integer $k$ that $G$ has a proper total $k$-coloring is the total chromatic number of $G$ and denoted by $\chi''(G)$. Behzad conjectured that  $\chi''(G)$ never exceeds $\Delta(G)+2$.\\ 
 There are some other types of simultaneous coloring in which we color at least two sets of the sets $V(G)$, $E(G)$, and $F(G)$ in the coloring \cite{borodin, chen, Ringel65, wang1, wang2}. Ringel conjectured \cite{Ringel65} that six colors are enough for simultaneous coloring of the vertices and faces of each plane graph $G$, i.e., $\chi_{vf}(G)\leq6$, and proved $\chi_{vf}(G)\leq7$. This conjecture was proved later in \cite{borodin2}. Kronk and Mitchem conjectured that $\Delta(G) +4$ colors are enough for entire coloring of the vertices, edges and faces of each plane graph $G$, i.e., $\chi_{vef}(G)\leq\Delta(G)+4$ and confirmed the conjecture for the case $\Delta(G)\leq3$ \cite{Kronk}. The conjecture was proved in \cite{borodin3} for $\Delta(G)\geq7$, in \cite{Sanders} for $\Delta(G)=6$, in \cite{wang2} for $4\leq\Delta(G)\leq 5$. Simultaneous coloring the edges and faces of 3- and 4-regular planar graphs has been considered by Jucovic \cite{Jucovic} and Fiamcik \cite{Fiamcik}. Melnikov (\cite{Melnikov}, page 543) conjectured that $\chi_{ef}(G)\leq \Delta(G)+3$ for any planar graph $G$. This conjecture was proved in \cite{borodin} for $\Delta(G)\geq10$. Precisely, Borodin proved that $\chi_{ef}(G)\leq\Delta(G)+1$ for each planar graph with $\Delta(G)\geq10$. Finally, this conjecture was proved by Adrian Waller for any planar graph \cite{Waller}.\\ 
In this paper, we are going to investigate the simultaneous coloring of the vertices and the incidences of a graph which is defined in the following.
\begin{definition}{\em{\cite{special}}}\label{verinccol}
Let $G$ be a graph. A $vi$-simultaneous proper $k$-coloring of $G$ is a coloring $c:V(G)\cup I(G)\longrightarrow[k]$ in which any two adjacent or incident elements in the set $V(G)\cup I(G)$ receive distinct colors. The $vi$-simultaneous chromatic number, denoted by $\chi_{vi}(G)$, is the smallest integer k such that $G$ has a $vi$-simultaneous proper $k$-coloring.
\end{definition}
For example,  in Figure~\ref{C4}, a $vi$-simultaneous proper $5$-coloring of the graph $G$ is shown.\\
The relation between  $vi$-simultaneous coloring of a graph and vertex coloring of its $\frac{3}{3}$ power was shown in \cite{special}. The fractional power of graphs was introduced by the second author in \cite{paper13}. By considering an arbitrary graph $G$ and positive integer $n$, we can construct two different graphs, named $n$-power of $G$, denoted by $G^n$, which is constructed by adding an edge between any two vertices of $G$ with distance at most $m$, and $n$-subdivision of $G$, denoted by $G^{\frac{1}{n}}$, which is constructed by replacing each edge $\{x,y\}$ of $G$ with a path of length $n$, including vertices $x=(xy)_0$, $(xy)_1$, $\ldots$, $(xy)_{n-1}$ and $y=(xy)_n$. The $\frac{m}{n}$-fractional power of a graph $G$ is defined to be the $m$-power of the $n$-subdivision of $G$. In other words, $G^{\frac{m}{n}}=(G^{\frac{1}{n}})^m$.\\
In \cite{special}, it is proved that for any graph $G$, $\chi_{vi}(G)=\chi(G^{\frac{3}{3}})$.

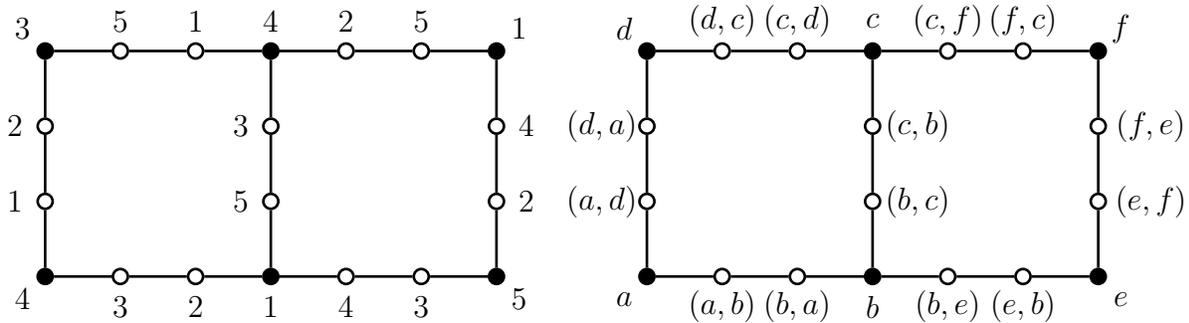
\begin{figure}[h]
\begin{center}
\begin{tikzpicture}[scale=1.0]
\tikzset{vertex/.style = {shape=circle,draw, line width=1pt, opacity=1.0, inner sep=2pt}}
\tikzset{vertex1/.style = {shape=circle,draw, fill=black, line width=1pt,opacity=1.0, inner sep=2pt}}
\tikzset{arc/.style = {->,> = latex', line width=1pt,opacity=1.0}}
\tikzset{edge/.style = {-,> = latex', line width=1pt,opacity=1.0}}
\node[vertex1] (a) at (0,0) {};
\node  at (0,-0.4) {$1$};
\node[vertex] (b) at (1,0) {};
\node  at (1,-0.4) {$4$};
\node[vertex] (c) at  (2,0) {};
\node  at (2,-0.4) {$3$};
\node[vertex1] (d) at  (3,0) {};
\node  at (3.3,-0.3) {$5$};
\node[vertex] (e) at  (3,1) {};
\node  at (3.4,1) {$2$};
\node[vertex] (f) at  (3,2) {};
\node  at (3.4,2) {$4$};
\node[vertex1] (g) at (3,3) {};
\node  at (3.3,3.3) {$1$};
\node[vertex] (h) at (2,3) {};
\node  at (2,3.4) {$5$};
\node[vertex] (i) at (1,3) {};
\node  at (1,3.4) {$2$};
\node[vertex1] (j) at (0,3) {};
\node  at (0,3.4) {$4$};
\node[vertex] (k) at (0,2) {};
\node  at (-0.4,2) {$3$};
\node[vertex] (m) at (0,1) {};
\node  at (-0.4,1) {$5$};
\node[vertex] (k01) at (-1,0) {};
\node  at (-1,-0.4) {$2$};
\node[vertex] (k2) at (-2,0) {};
\node  at (-2,-0.4) {$3$};
\node[vertex1] (k3) at (-3,0) {};
\node  at (-3.3,-0.3) {$4$};
\node[vertex] (k4) at (-3,1) {};
\node  at (-3.4,1) {$1$};
\node[vertex] (k5) at (-3,2) {};
\node  at (-3.4,2) {$2$};
\node[vertex1] (k6) at (-3,3) {};
\node  at (-3.3,3.3) {$3$};
\node[vertex] (k7) at (-2,3) {};
\node  at (-2,3.4) {$5$};
\node[vertex] (k8) at (-1,3) {};
\node  at (-1,3.4) {$1$};
\node[vertex] (m01) at (9,0) {};
\node  at (9,-0.4) {$(b,e)$};
\node[vertex] (m2) at (10,0) {};
\node  at (10,-0.4) {$(e,b)$};
\node[vertex1] (m3) at (11,0) {};
\node  at (11.3,-0.3) {$e$};
\node[vertex] (m4) at (11,1) {};
\node  at (11.7,1) {$(e,f)$};
\node[vertex] (m5) at (11,2) {};
\node  at (11.7,2) {$(f,e)$};
\node[vertex1] (m6) at (11,3) {};
\node  at (11.3,3.3) {$f$};
\node[vertex] (m7) at (10,3) {};
\node  at (10,3.4) {$(f,c)$};
\node[vertex] (m8) at (9,3) {};
\node  at (9,3.4) {$(c,f)$};

\draw[edge] (a)  to (b);
\draw[edge] (b)  to (c);
\draw[edge] (c)  to (d);
\draw[edge] (d)  to (e);
\draw[edge] (e)  to (f);
\draw[edge] (f)  to (g);
\draw[edge] (g)  to (h);
\draw[edge] (h)  to (i);
\draw[edge] (i)  to (j);
\draw[edge] (j)  to (k);
\draw[edge] (k)  to (m);
\draw[edge] (m)  to (a);

\draw[edge] (a)  to (k01);
\draw[edge] (k01)  to (k2);
\draw[edge] (k2)  to (k3);
\draw[edge] (k3)  to (k4);
\draw[edge] (k4)  to (k5);
\draw[edge] (k5)  to (k6);
\draw[edge] (k6)  to (k7);
\draw[edge] (k7)  to (k8);
\draw[edge] (k8)  to (j);
\node[vertex1] (a1) at (5,0) {};
\node  at (4.7,-0.3) {$a$};
\node[vertex] (b1) at (6,0) {};
\node  at (6,-0.4) {$(a,b)$};
\node[vertex] (c1) at  (7,0) {};
\node  at (7,-0.4) {$(b,a)$};
\node[vertex1] (d1) at  (8,0) {};
\node  at (8,-0.4) {$b$};
\node[vertex] (e1) at  (8,1) {};
\node  at (8.6,1) {$(b,c)$};
\node[vertex] (f1) at  (8,2) {};
\node  at (8.6,2) {$(c,b)$};
\node[vertex1] (g1) at (8,3) {};
\node  at (8,3.4) {$c$};
\node[vertex] (h1) at (7,3) {};
\node  at (7,3.4) {$(c,d)$};
\node[vertex] (i1) at (6,3) {};
\node  at (6,3.4) {$(d,c)$};
\node[vertex1] (j1) at (5,3) {};
\node  at (4.7,3.3) {$d$};
\node[vertex] (k1) at (5,2) {};
\node  at (4.4,2) {$(d,a)$};
\node[vertex] (m1) at (5,1) {};
\node  at (4.4,1) {$(a,d)$};
\draw[edge] (a1)  to (b1);
\draw[edge] (b1)  to (c1);
\draw[edge] (c1)  to (d1);
\draw[edge] (d1)  to (e1);
\draw[edge] (e1)  to (f1);
\draw[edge] (f1)  to (g1);
\draw[edge] (g1)  to (h1);
\draw[edge] (h1)  to (i1);
\draw[edge] (i1)  to (j1);
\draw[edge] (j1)  to (k1);
\draw[edge] (k1)  to (m1);
\draw[edge] (m1)  to (a1);
\draw[edge] (d1)  to (m01);
\draw[edge] (m01)  to (m2);
\draw[edge] (m2)  to (m3);
\draw[edge] (m3)  to (m4);
\draw[edge] (m4)  to (m5);
\draw[edge] (m5)  to (m6);
\draw[edge] (m6)  to (m7);
\draw[edge] (m7)  to (m8);
\draw[edge] (m8)  to (g1);
\end{tikzpicture}
\caption{(Right) Black vertices are corresponding to the vertices of $G$ and white vertices are corresponding to the incidences of $G$. The incidence $(u,\{u,v\})$ is denoted by $(u,v)$. (Left) $vi$-simultaneous proper $5$-coloring of graph $G$.}
\label{C4}
\end{center}
\end{figure}
Let $G$ be a graph with $\Delta(G)=\Delta$. We need at least $\Delta+2$ colors for any $vi$-simultaneous proper coloring of $G$. Suppose that $v$ is a $\Delta$-vertex in $G$ and $u\in N_G(v)$. Since any two elements of $I_1^G[v]\cup\{(u,v)\}$ are incident, $\chi_{vi}(G)\geq\Delta(G)+2$. Therefore, the problem of finding some upper bounds for $vi$-simultaneous chromatic number of graph $G$ in terms of $\Delta(G)$ is a natural problem.\\
We can define some special kind of $vi$-simultaneous coloring of graphs by restricting the number of colors that appear on the incidences of each vertex.
\begin{definition}{\em{\cite{special}}}\label{(k,l)IncidenceCol}
	A $vi$-simultaneous proper $k$-coloring of a graph $G$ is called $vi$-simultaneous $(k,s)$-coloring of $G$ if for any vertex $v$, the number of colors used for coloring $I_2(v)$ is at most $s$. We denote by $\chi_{vi,s}(G)$ the smallest number of colors required for a $vi$-simultaneous $(k,s)$-coloring of $G$.
\end{definition}
As you can see in Figure~\ref{C4}, the given coloring for the graph is a $vi$-simultaneous $(5,1)$-coloring. Obviously, one can show that $\chi_{vi,1}(G)\geq \chi_{vi,2}(G)\geq \chi_{vi,3}(G)\geq\cdots\geq\chi_{vi,\Delta}(G)=\chi_{vi}(G)$ for any graph $G$ with $\Delta(G)=\Delta$. For example $\chi_{vi,1}(K_3)=6>\chi_{vi,2}(K_3)=\chi_{vi}(K_3)=5$ (see Figure ~\ref{Fig2}).\\
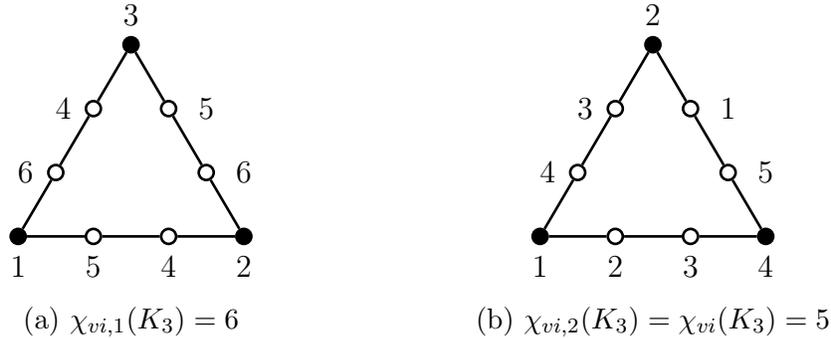
\begin{figure}[h]
	\begin{subfigure}{.55\textwidth}
		\begin{center}
			\begin{tikzpicture}[scale=1.0]
			\tikzset{vertex/.style = {shape=circle,draw, line width=1pt, opacity=1.0, inner sep=2pt}}
			\tikzset{vertex1/.style = {shape=circle,draw, fill=black, line width=1pt,opacity=1.0, inner sep=2pt}}
			\tikzset{arc/.style = {->,> = latex', line width=1pt,opacity=1.0}}
			\tikzset{edge/.style = {-,> = latex', line width=1pt,opacity=1.0}}
			\node[vertex1] (a) at (0,0) {};
			\node  at (0,-0.4) {$1$};
			\node[vertex] (b) at (1,0) {};
			\node  at (1,-0.4) {$5$};
			\node[vertex] (c) at  (2,0) {};
			\node  at (2,-0.4) {$4$};
			\node[vertex1] (d) at  (3,0) {};
			\node  at (3,-0.4) {$2$};
			\node[vertex] (e) at  (2.5,0.85) {};
			\node  at (3,0.85) {$6$};
			\node[vertex] (f) at  (2,1.7) {};
			\node  at (2.5,1.7) {$5$};
			\node[vertex1] (g) at (1.5,2.55) {};
			\node  at (1.5,2.95) {$3$};
			\node[vertex] (h) at (1,1.7) {};
			\node  at (0.6,1.7) {$4$};
			\node[vertex] (i) at (0.5,0.85) {};
			\node  at (0.1,0.85) {$6$};
			\draw[edge] (a)  to (b);
			\draw[edge] (b)  to (c);
			\draw[edge] (c)  to (d);
			\draw[edge] (d)  to (e);
			\draw[edge] (e)  to (f);
			\draw[edge] (f)  to (g);
			\draw[edge] (g)  to (h);
			\draw[edge] (h)  to (i);
			\draw[edge] (i)  to (a);
			\end{tikzpicture}
			\caption{$\chi_{vi,1}(K_3)=6$}\label{1}
		\end{center}
	\end{subfigure}
	\begin{subfigure}{.3\textwidth}
		\begin{center}
			\begin{tikzpicture}[scale=1.0]
			\tikzset{vertex/.style = {shape=circle,draw, line width=1pt, opacity=1.0, inner sep=2pt}}
			\tikzset{vertex1/.style = {shape=circle,draw, fill=black, line width=1pt,opacity=1.0, inner sep=2pt}}
			\tikzset{arc/.style = {->,> = latex', line width=1pt,opacity=1.0}}
			\tikzset{edge/.style = {-,> = latex', line width=1pt,opacity=1.0}}
			\node[vertex1] (a) at (0,0) {};
			\node  at (0,-0.4) {$1$};
			\node[vertex] (b) at (1,0) {};
			\node  at (1,-0.4) {$2$};
			\node[vertex] (c) at  (2,0) {};
			\node  at (2,-0.4) {$3$};
			\node[vertex1] (d) at  (3,0) {};
			\node  at (3,-0.4) {$4$};
			\node[vertex] (e) at  (2.5,0.85) {};
			\node  at (3,0.85) {$5$};
			\node[vertex] (f) at  (2,1.7) {};
			\node  at (2.5,1.7) {$1$};
			\node[vertex1] (g) at (1.5,2.55) {};
			\node  at (1.5,2.95) {$2$};
			\node[vertex] (h) at (1,1.7) {};
			\node  at (0.6,1.7) {$3$};
			\node[vertex] (i) at (0.5,0.85) {};
			\node  at (0.1,0.85) {$4$};
			\draw[edge] (a)  to (b);
			\draw[edge] (b)  to (c);
			\draw[edge] (c)  to (d);
			\draw[edge] (d)  to (e);
			\draw[edge] (e)  to (f);
			\draw[edge] (f)  to (g);
			\draw[edge] (g)  to (h);
			\draw[edge] (h)  to (i);
			\draw[edge] (i)  to (a);
			\end{tikzpicture}
			\caption{$\chi_{vi,2}(K_3)=\chi_{vi}(K_3)=5$}\label{2}
		\end{center}
	\end{subfigure}
	\caption{(Left) $vi$-simultaneous proper $(6,1)$-coloring of $K_3$. (Right) $vi$-simultaneous proper $(5,2)$-coloring of $K_3$.}\label{Fig2}
\end{figure}
In \cite{mahsa, special}, $vi$-simultaneous proper coloring of graphs is investigated. In \cite{mahsa}, it is proved that $\chi_{vi}(G)\leq 9$ for any graph $G$ with $\Delta(G)\leq 4$. Also, the following conjecture was proposed:
\begin{conjecture}{\em{\cite{mahsa}}}\label{cmahsa}
Let $G$ be a graph with $\Delta(G)\geq 2$. Then $\chi_{vi}(G)\leq 2\Delta(G)+1$.
\end{conjecture}
The correctness of the conjecture is proved for some classes of graphs in \cite{special}. In this paper, we are going to find the upper bounds for the $vi$-simultaneous chromatic number of outerplanar graphs and investigate the correctness of the conjecture for these graphs. The main theorems are as follows.\\
\begin{theorem}\label{1}
	If $G$ is an outerplanar graph, then $\chi_{vi,2}(G)\leq \Delta(G)+3$.
\end{theorem}
\begin{theorem}\label{2}
	Suppose that  $G$ is an outerplanar graph. Then
	$$\chi_{vi,1}(G)\leq \left\{
	\begin{array}{cc}
	\Delta+3 &  \Delta(G)\geq 4\ and\  g(G)\geq4,\\
	\Delta+2 & \Delta(G)\geq4\  and\ g(G)\geq6 ,\\
	\Delta+2 & \Delta(G)\geq5\ and\  g(G)\geq4.
	\end{array}
	\right.
	$$
	Note that, since $\chi_{vi,1}(G)\geq \Delta+2$, in the last two cases, we have $\chi_{vi,1}(G)=\Delta+2$.
\end{theorem}
The paper is organized as follows. In Section 2, some preliminary definitions and theorems are mentioned and Section 3 is devoted to the proofs of the main theorems.
\section{Preliminaries and Definitions}
To prove theorems expressed in Section $3$, we need some definitions and theorems which are stated in \cite{paper13, special}. As we mentioned before, the $vi$-simultaneous coloring of some classes of graphs are investigated in \cite{special}. Some of these results we are going to use are as follows.
\begin{theorem}{\em{\cite{paper13}}}\label{cycle-path}
	Let $m,n \in \mathbb{N}$ and $k\geq 3$.
\begin{itemize}
	\item 
	$\chi(C_k^m)=\left\{
	\begin{array}{cc}
		k & m\geq \lfloor\frac{k}{2}\rfloor\\
		\lceil\frac{k}{\lfloor\frac{k}{m+1}\rfloor}\rceil & m<\lfloor\frac{k}{2}\rfloor
	\end{array}
	\right.$
	\item $\chi(P_k^m)=\min\{m+1,k\}$
\end{itemize}
\end{theorem}
By Theorem~\ref{cycle-path} and the fact that $C_n^{\frac{3}{3}}=C_{3n}^{3}$ and $P_n^{\frac{3}{3}}=P_{3n-2}^{3}$, we have the following corollary. 
\begin{corollary}\label{cycleandpath}
If $G$ is a cycle of length $n$, then  $\chi_{vi}(C_n)=\chi(C_n^{\frac{3}{3}})=4$, when $n\equiv 0$ $(\hspace{-.23cm}\mod 4)$. Otherwise, $\chi_{vi}(C_n)=\chi(C_n^{\frac{3}{3}})=5$. Moreover, for any path $P_n$, we have $\chi_{vi}(P_n)=\chi(P_n^{\frac{3}{3}})=4$.
\end{corollary}
\begin{theorem}{\em{\cite{special}}}\label{cycles}
	Let $3\leq n\in\mathbb{N}$. Then
	\[\chi_{vi,1}(C_n)=\left\{\begin{array}{lll} 6 & n=3,\\ 4 & n\equiv 0\ (mod\ 4),\\ 5 & otherwise. \end{array}\right.\]
\end{theorem}
\begin{theorem}{\em{\cite{special}}}\label{tree}
	Let $F$ be a forest. Then
	\[\chi_{vi,1}(F)=\left\{\begin{array}{lll} 1 & \Delta(F)=0,\\ 4 & \Delta(F)=1,\\ \Delta(F)+2 & \Delta(F)\geq2. \end{array}\right.\]
\end{theorem}
\begin{theorem}{\em{\cite{special}}}\label{complete}
	$\chi_{vi}(K_n)=n+2$ for each $n\in\mathbb{N}\setminus\{1\}$.
\end{theorem}
A graph $G$ is a $k$-degenerated graph if any subgraph of $G$ contains a vertex of degree at most $k$. 
\begin{theorem}{\em{\cite{special}}}\label{kdegenerated}
	Let $k\in\mathbb{N}$ and $G$ be a $k$-degenerated graph with $\Delta(G)\geq2$. Then $\chi_{vi,k}(G)\leq \Delta(G)+2k$. 
\end{theorem}
\begin{theorem}{\em{\cite{special}}}\label{firstlem}
	Let $G$ be a graph with maximum degree $\Delta$ and $c$ is a proper $(\Delta+2)$-coloring of $G^{\frac{3}{3}}$ with colors from $[\Delta+2]$. Then $|c(I_2(v))|\leq\Delta-d_G(v)+1$ for any $t$-vertex $v$, where $c(I_2(v))=\{c(a)\ |\ a\in I_2(v)\}$ . Specially $|c(I_2(v))|=1$ for any $\Delta$-vertex $v$ of $G$.
\end{theorem}
\begin{theorem}{\em{\cite{special}}}\label{secondlem}
	Let $G$ be a graph, $e$ be a cut edge of $G$ and $C_1$ and $C_2$ be two components of $G-e$. Then $\chi_{vi,l}(G)=\max\{\chi_{vi,l}(H_1),\chi_{vi,l}(H_2)\}$ where $H_i=C_i+e$ for $i\in\{1,2\}$ and $1\leq l\leq\Delta(G)$.
\end{theorem}
\begin{theorem}{\em{\cite{special}}}\label{thirdlem}
	Let $G_1$ and $G_2$ be two graphs, $V(G_1)\cap V(G_2)=\{v\}$ and $G=G_1\cup G_2$. Then
	\[\chi_{vi,1}(G)=\max\{\chi_{vi,1}(G_1),\chi_{vi,1}(G_2), \deg_G(v)+2\}.\]
\end{theorem}
\begin{corollary}\label{blocks}
	Let $k\in\mathbb{N}$ and $G$ be a graph with blocks $B_1,\ldots,B_k$. Then
	\[\chi_{vi,1}(G)=\max\{\chi_{vi,1}(B_1),\ldots,\chi_{vi,1}(B_k), \deg_G(v_1)+2,\ldots,\deg_G(v_s)+2\},\]
	where $v_1,\ldots,v_s$ are the cut vertices of the graph $G$.
\end{corollary}

In \cite{special}, it is shown that there is a relationship between $vi$-simultaneous coloring of graphs and two parameters $\chi''(G)$ and $st(G)$, where $st(G)$ is the star arboricity of the graph $G$ which was introduced by Algor and Alon \cite{star}.  The star arboricity of a graph is the minimum number of star forests (forests whose connected components are stars) in $G$ whose union covers all edges of $G$. 
\begin{theorem}{\em{\cite{special}}}\label{start1}
	For any graph $G$, we have $\chi_{vi}(G)\leq \chi''(G)+st(G)$.
\end{theorem}
\begin{theorem}{\em{\cite{special}}}\label{upperbound-list-vi1}
	$\chi_{vi,1}(G)\leq\max\{\chi(G^2),\chi_{l}(G)+\Delta(G)+1\}$ for any nonempty graph $G$. Specially, if $\chi(G^2)\geq\chi_{l}(G)+\Delta(G)+1$, then $\chi_{vi,1}(G)=\chi(G^2)$.
\end{theorem}
\section{Main Theorems}\label{sec5}
In this section, we investigate the $vi$-simultaneous chromatic number of outerplanar graphs with defined maximum degree and girth. From now on, black vertices are corresponded to the $t$-vertices and white vertices are corresponded to the $i$-vertices of a graph in the figures.\\
A graph is said to be planar, if it can be drawn in the plane so that its edges intersect only at their ends. An outerplanar graph is a graph that has a planar drawing for which all vertices belong to the outer face of the drawing and in a 2-conneced outerplanar graph, the outer face is a hamiltonian cycle. The girth of a graph $G$, denoted by $g(G)$, is the length of a shortest cycle in $G$. A block of the graph $G$ is a maximal 2-connected subgraph of $G$ and if $G$ is 2-connected, then $G$ itself is called a block. Also, an end block of $G$ is a block with a single cut vertex. We know that each connected graph with at least one cut vertex has at least two end blocks.\\
A face $f \in F(G)$ with degree $k$ is denoted by its boundary walk $f=[v_1v_2\ldots v_k]$, where $v_1,v_2,\ldots,v_k$ are its vertices in the clockwise order. We say that a face $f=[v_1v_2\ldots v_k]$ is an {\it{end face}} of an outerplane graph $G$, if $v_2,\ldots,v_{k-1}$ are all 2-vertices in $G$.\\
We know that any outerplanar graph is a 2-degenerated graph. So by use of Theorem \ref{kdegenerated} we can achieve the following upper bound for $\chi_{vi,2}(G)$.  
\begin{proposition}\label{delta+4}
	If $G$ is an outerplanar graph with maximum degree $\Delta$, then $\chi_{vi,2}(G)\leq\Delta+4$. 
\end{proposition}
In continue we improve this upper bound to $\Delta(G)+3$ but we need the Theorem \ref{delta+4} in the proof of Theorem \ref{delta>3}. The following lemma easily follows from Proposition 6.1.20 in \cite{west2}.
\begin{lemma}\label{lemmaouterplanar}
	For any nonempty outerplanar graph $G$ at least one of the following holds:\\
	(i) $G$ has a 1-vertex,\\
	(ii) $G$ has two adjacent 2-vertices,\\
	(iii) $G$ has a $2$-vertex with adjacent neighbors.\\
	Especially, any 2-connected outerplanar graph holds in (ii) or (iii).
\end{lemma}
\subsection{Proof of the Theorem~\ref{1}}
In oreder to prove of the Theorem~\ref{1}, first we are going to prove it for 2-connected outerplanar graphs with maximum degree at least $3$. 
Before we get into that, it should be noted that by Corollary~\ref{cycleandpath}, Theorem~\ref{1} is true for all outerplanar graphs with maximum degree at most two.\\
A graph $G$ is called subcubic, if $\Delta(G)\leq3$. Also, a 3-regular graph is called a cubic graph. 
\begin{theorem}\label{2connectedout}
	If $G$ is a 2-connected outerplanar graph with $\Delta(G)=3$, then $\chi_{vi,2}(G)\leq 6$. Moreover, the bound is tight.
\end{theorem}
\begin{proof}{
		We prove the theorem by induction on the order of $G$. There is only one 2-connected outerplanar graphs of order $4$  with $\Delta(G)=3$ which is $vi$-simultaneous $(6,1)$-colorable (See Figure~\ref{K4-e}). Easily one can show that 6 colors are necessary for this graph.
		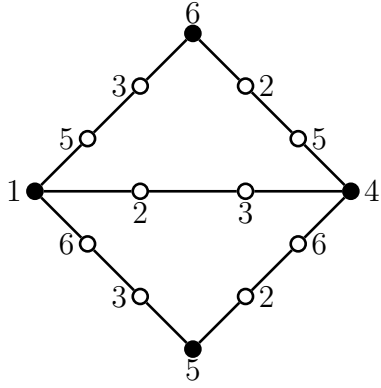
\begin{figure}[h]
			\begin{center}
				\begin{tikzpicture}[scale=0.7]
				\tikzset{vertex/.style = {shape=circle,draw, line width=1pt, opacity=1.0, inner sep=2pt}}
				\tikzset{vertex1/.style = {shape=circle,draw, fill=black, line width=1pt,opacity=1.0, inner sep=2pt}}
				\tikzset{edge/.style = {-,> = latex', line width=1pt,opacity=1.0}}
				\node[vertex1] (a1) at (-3,0) {};
				\node  at (-3.4,0) {$1$};
				\node[vertex] (a2) at (-2,1) {};
				\node  at (-2.4,1) {$5$};
				\node[vertex] (a3) at (-1,2) {};
				\node  at (-1.4,2) {$3$};
				\node[vertex1] (a4) at (0,3) {};
				\node  at (0,3.4) {$6$};
				\node[vertex] (a5) at (1,2) {};
				\node  at (1.4,2) {$2$};
				\node[vertex] (a6) at (2,1) {};
				\node  at (2.4,1) {$5$};
				\node[vertex1] (a7) at (3,0) {};
				\node  at (3.4,0) {$4$};
				\node[vertex] (a8) at (2,-1) {};
				\node  at (2.4,-1) {$6$};
				\node[vertex] (a9) at (1,-2) {};
				\node  at (1.4,-2) {$2$};
				\node[vertex1] (a10) at (0,-3) {};
				\node  at (0,-3.4) {$5$};
				\node[vertex] (a11) at (-1,-2) {};
				\node  at (-1.4,-2) {$3$};
				\node[vertex] (a12) at (-2,-1) {};
				\node  at (-2.4,-1) {$6$};
				\node[vertex] (b1) at (-1,0) {};
				\node  at (-1,-0.4) {$2$};
				\node[vertex] (b2) at (1,0) {};
				\node  at (1,-0.4) {$3$};
				\draw[edge] (a1) -- (a2) -- (a3) -- (a4) -- (a5) -- (a6) -- (a7) -- (a8) -- (a9) -- (a10) -- (a11) -- (a12) -- (a1) -- (b1) -- (b2) -- (a7);
				\end{tikzpicture}
				\caption{$vi$-simultaneous $(6,1)$-coloring 2-connected outerplanar graph of order four.}
				\label{K4-e}
			\end{center}
		\end{figure}
		Now suppose that $G$ is a 2-connected outerplanar graph of order $n\geq 5$ with $\Delta(G)=3$, and the statement is true for all 2-connected outerplanar graphs with maximum degree 3 of order less than $n$. Since $G$ is 2-connected and $\Delta(G)=3$, $G$ has an end face $f=[v_iv_{i+1} \ldots v_j]$ of degree at least 3.\\
		First, suppose that $f=[v_iv_{i+1}v_{i+2}]$ and $G'=G-v_{i+1}$. If $\Delta(G')=2$, then by use of Theorem \ref{cycles}, $G'$ has a proper $vi$-simultaneous $(5,1)$-coloring, named $c$ with colors in $[5]$. If $\Delta(G')=3$, then by the  induction hypothesis, $G'$ has a proper $vi$-simultaneous $(6,2)$-coloring $c$ with colors in $[6]$. Now, it suffices to extend $c$ to a proper $vi$-simultaneous $(6,2)$-coloring for $G$.\\
		Since $v_i$, $v_{i+2}$, $(v_i,v_{i+2})$ and $(v_{i+2},v_i)$ are pairwise adjacent in $(G')^{\frac{3}{3}}$, they must have different colors. Without loss of generality, suppose that $c(v_i)=1$, $c((v_i,v_{i+2}))=2$, $c((v_{i+2},v_i))=3$ and $c(v_{i+2})=4$. If at least one of colors $1$ or $4$ is available for one of the $i$-vertices $(v_i,v_{i+1})$, $(v_{i+1},v_i)$, $(v_{i+1},v_{i+2})$ and $(v_{i+2},v_{i+1})$, then with a simple review $c$ can be extended to a $(6,2)$-proper coloring $c'$ for $G^{\frac{3}{3}}$. In addition, if $c((v_{i-1},v_i))=4$, then color 4 is available for $(v_{i+1},v_i)$ and if $c((v_{i+3},v_{i+2}))=1$, then color 1 is available for $(v_{i+1},v_{i+2})$.  Otherwise, the $i$-vertices $(v_i,v_{i-1})$ and $(v_{i+2},v_{i+3})$ must have colors $4$ and $1$, respectively. If at least one of colors $5$ or $6$ is available for $i$-vertices $(v_i,v_{i+1})$ and $(v_{i+2},v_{i+1})$, then color two $i$-vertices $(v_i,v_{i+1})$ and $(v_{i+2},v_{i+1})$ with their available colors in $\{5,6\}$. Also, assign colors $2$ and $3$ to $i$-vertices $(v_{i+1},v_{i+2})$ and $(v_{i+1},v_{i})$, respectively and color $v_{i+1}$ with color $6$. Otherwise, the $i$-vertices $(v_{i-1},v_{i})$ and $(v_{i+3},v_{i+2})$ must have colors $5$ and $6$. Without loss of generality, let $c((v_{i-1},v_{i}))=5$ and $c((v_{i+3},v_{i+2}))=6$. Now we change the color of $(v_i,v_{i+2})$ to 6 and then we use color 2 for the $i$-vertices $(v_i,v_{i+1})$ and $(v_{i+2},v_{i+1})$, color 3 for $(v_{i+1},v_i)$, color 6 for $(v_{i+1},v_{i+2})$ and color 5 for $v_{i+1}$. Easily the extended coloring is a $vi$-simultaneous $(6,2)$-coloring for $G$.\\
Now suppose that $f=[v_iv_{i+1}v_{i+2}v_{i+3}]$. There are at least one available color for $i$-vertices $(v_i,v_{i+1})$ and $(v_{i+3},v_{i+2})$. In addition, we color vertices $(v_{i+1},v_{i})$ and $v_{i+2}$ with color 3, the vertices $v_{i+1}$ and $(v_{i+2},v_{i+3})$ with color 2, the vertex $(v_{i+1},v_{i+2})$ with color 1 and the vertex $(v_{i+2},v_{i+1})$ with color 4. Easily one can show that, this coloring is a $vi$-simultaneous $(6,2)$-coloring for $G$.\\
If $f=[v_iv_{i+1}v_{i+2}\ldots v_{j}]$ and $j\geq i+4$, then we consider the graph $G'=G+e$ where $e=\{v_{i+1},v_{i+3}\}$ and similar to the first case, we prove that $\chi_{vi,2}(G')\leq6$ which concludes that $\chi_{vi,2}(G)\leq6$.
}\end{proof}
\begin{theorem}\label{delta3}
	For any outerplanar graph $G$ with maximum degree $3$, $\chi_{vi,2}(G)\leq 6$.
\end{theorem}
\begin{proof}{
		We prove the theorem by the induction on the number of blocks. According to Theorem \ref{2connectedout}, the claim is true for blocks. Now suppose that $k\in\mathbb{N}\setminus\{1\}$ and $G$ is an outerplanar graph with $\Delta(G)=3$ and $k$ blocks and the statement is true for all outerplanar graphs with $\Delta=3$ and less than $k$ blocks. Since $\Delta(G)=3$, from any two blocks with a common cut vertex, one block must be $K_2$. So $G$ has at least one cut edge such as $e=\{u,v\}$. If $d_G(u)\neq1\neq d_G(v)$, then by use of Theorem \ref{secondlem}, the theorem follows by induction applied to $C_1+e$ and $C_2+e$ where $C_1$ and $C_2$ are the components of $G-e$. Now suppose that $d_G(u)=1<d_G(v)$. In this case, the graph $G-u$ is an outerplanar graph with $k-1$ blocks. Hence, by induction hypothesis, $G-u$ has a $vi$-simultaneous $(6,2)$-coloring, named $c$, with color set $[6]$. Suppose that $N_G(v)=\{v_1,v_2,u\}$. Now, color $i$-vertices $(v,u)$ with one color from $[6]\setminus\{c(v),c((v,v_1)),c((v_1,v)),c((v,v_2)),c((v_2,v))\}$, then color the $i$-vertex $(u,v)$ with the color $c((v_1,v))$ or $c((v_2,v))$ and finally color the vertex $u$ with one available color from $[6]\setminus\{c(v),c((v,u)),c((u,v))\}$. The given coloring is a $vi$-simultaneous $(6,2)$-coloring for $G$.
}\end{proof}
\begin{theorem}\label{delta>3}
	For any outerplanar graph $G$ with $\Delta(G)=\Delta\geq4$, $\chi_{vi,2}(G)\leq\Delta+3$.
\end{theorem}
\begin{figure}[h]
	\begin{center}
		\begin{tikzpicture}[scale=0.8]
		\tikzset{vertex/.style = {shape=circle,draw, fill=black, line width=1pt,opacity=1.0, inner sep=2pt}}
		\tikzset{vertex1/.style = {shape=circle,draw, fill=white, line width=1pt,opacity=1.0, inner sep=2pt}}
		\tikzset{edge/.style = {-,> = latex', line width=1pt,opacity=1.0}}
		\node[vertex] (a1) at (0,0) {};
		\node () at  (0,0.4) {$3$};
		\node[vertex1] (a2) at (1,0) {};
		\node () at  (1,0.4) {$6$};
		\node[vertex1] (a3) at  (2,0) {};
		\node () at  (2,0.4) {$5$};
		\node[vertex] (a4) at  (3,0) {};
		\node () at  (3.4,-0.4) {$1$};
		\node[vertex1] (a5) at  (4,0) {};
		\node () at  (4,0.4) {$3$};
		\node[vertex1] (a6) at  (5,0) {};
		\node () at  (5,0.4) {$6$};
		\node[vertex] (a7) at (6,0) {};
		\node () at  (6.4,0) {$5$};
		\node[vertex] (a8) at (3,3) {};
		\node () at  (3,3.4) {$4$};
		\node[vertex1] (a9) at (3,2) {};
		\node () at  (2.6,2) {$6$};
		\node[vertex1] (a10) at (3,1) {};
		\node () at  (2.6,1) {$2$};
		\node[vertex1] (a11) at (3,-1) {};
		\node () at  (2.6,-1) {$4$};
		\node[vertex1] (a12) at (3,-2) {};
		\node () at  (2.6,-2) {$6$};
		\node[vertex] (a13) at (3,-3) {};
		\node () at  (3,-3.4) {$2$};
		\node[vertex1] (a14) at (1,-1) {};
		\node () at  (0.6,-1) {$4$};
		\node[vertex1] (a15) at (2,-2) {};
		\node () at  (1.6,-2) {$5$};
		\node[vertex1] (a16) at (4,-2) {};
		\node () at  (4.4,-2) {$3$};
		\node[vertex1] (a17) at (5,-1) {};
		\node () at  (5.4,-1) {$4$};
		\node[vertex1] (a18) at (5,1) {};
		\node () at  (5.4,1) {$2$};
		\node[vertex1] (a19) at (4,2) {};
		\node () at  (4.4,2) {$3$};
		\node () at  (3,-4) {$G_1$};
		\draw[edge] (a13) -- (a15) -- (a14) -- (a1) -- (a2) -- (a3) -- (a4) -- (a11) -- (a12) -- (a13) -- (a16) -- (a17) -- (a7) -- (a6) -- (a5) -- (a4) -- (a10) -- (a9) -- (a8) -- (a19) -- (a18) -- (a7);
		\node[vertex] (b1) at (9,2) {};
		\node[vertex] (b2) at (8,2) {};
		\node[vertex] (b3) at  (9,1) {};
		\node[vertex] (b4) at  (10,2) {};
		\node[vertex] (b5) at  (9,3) {};
		\node[vertex] (b6) at  (9,-2) {};
		\node[vertex] (b7) at (8,-2) {};
		\node[vertex] (b8) at (9,-3) {};
		\node[vertex] (b9) at (10,-2) {};
		\node[vertex] (b10) at (9,-1) {};
		\node[vertex] (b11) at (12,2) {};
		\node[vertex] (b12) at (11,2) {};
		\node[vertex] (b13) at (12,1) {};
		\node[vertex] (b14) at (13,2) {};
		\node[vertex] (b15) at (12,3) {};
		\node[vertex] (b16) at (12,-2) {};
		\node[vertex] (b17) at (11,-2) {};
		\node[vertex] (b18) at (12,-3) {};
		\node[vertex] (b19) at (13,-2) {};
		\node[vertex] (b20) at (12,-1) {};
		\draw[edge] (b1) -- (b2) -- (b1) -- (b3) -- (b1) -- (b4) -- (b1) -- (b5);
		\draw[edge] (b10) -- (b6) -- (b7) -- (b8) -- (b6) -- (b9);
		\draw[edge] (b15) -- (b11) -- (b12) -- (b13) -- (b11) -- (b14) -- (b15);
		\draw[edge] (b20) -- (b16) -- (b17) -- (b18) -- (b19) -- (b16) -- (b18);
		\end{tikzpicture}
		\caption{Five outerplanar graphs of order $5$ with maximum degree $4$. Uncolored graphs are subgraphs of the colored graph.}
		\label{out5}
	\end{center}
\end{figure}
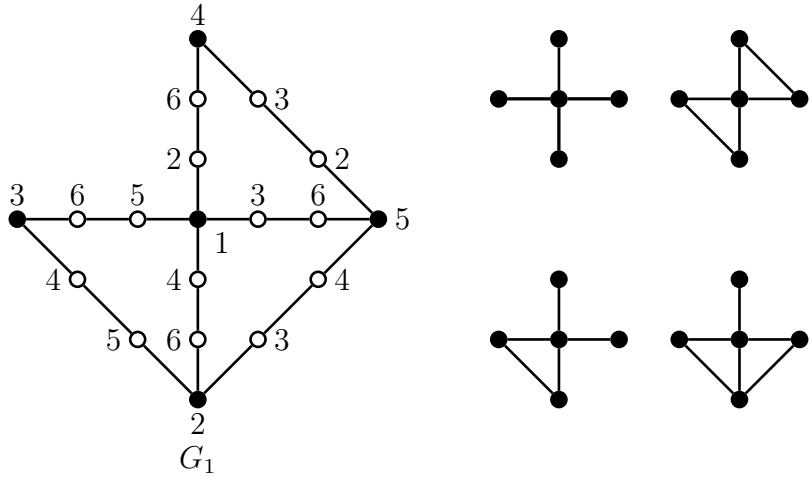
\begin{proof}{
We prove the theorem by induction on the order of $G$. Since $\Delta(G)\geq4$, $G$ has at least five vertices. There are five outerplanar graphs of order five with maximum degree at least 4 (Figure \ref{out5}). Since all of them are subgraphs of the first graph $G_1$, we need only to find a $vi$-simultaneous $(7,2)$-coloring for $G_1$. A $vi$-simultaneous $(6,1)$-coloring of $G_1$ is shown in Figure \ref{out5}. Now suppose that the theorem is true for all outerplanar graphs of order less than $n$ and let $G$ be an outerplanar graph of order $n$ with $\Delta\geq4$. There are three cases, corresponding to the cases in Lemma \ref{lemmaouterplanar}.\\
		(i) Let $d_G(v)=1$ and $u$ be the vertex adjacent to $v$. Then $G'=G-v$ is an outerplanar graph of smaller order and maximum degree at most $\Delta$. If  $\Delta(G)=4$ and $\Delta(G')=\Delta(G)-1=3$ then by use of Theorem \ref{delta+4} and otherwise, by the induction hypothesis, $G'$ has a $vi$-simultaneous $(\Delta+3,2)$-coloring named $c$. We extend $c$ to a $vi$-simultaneous $(\Delta+3,2)$-coloring of $G$. The degree of $u$ in $G'$ is at most $\Delta-1$, so there are at most $\Delta+2$ colors used by the vertices of $I[u]$. Hence, there is at least one color left to color the $i$-vertex $(u,v)$. Also, the $i$-vertex $(v,u)$ can be colored by one color of $c(I_2(u))$. Finally, there are at least $\Delta$ available colors for coloring $v$.\\
		(ii) $G$ has two adjacent 2-vertices $v$ and $u$. Let $x$ be the vertex adjacent to $v$ and $y$ be the vertex adjacent to $u$. Consider $G'=G-\{v,u\}$. Again, $G'$ is outerplanar with smaller order and maximum degree at most $\Delta$ and so by use of Theorem \ref{delta+4} or the induction hypothesis, $G$ has a $vi$-simultaneous $(\Delta+3,2)$-coloring named $c$. We extend $c$ to a $vi$-simultaneous $(\Delta+3,2)$-coloring of $G$. The degree of $x$ and $y$ in $G'$ is at most $\Delta-1$, so similar to the previous case, there is at least one color left to color each of the $i$-vertices $(x,v)$ and $(y,u)$. Also, each of the $i$-vertices $(v,x)$ and $(u,y)$ can be colored by one color of $c(I^{G'}_2(x))$ and $c(I^{G'}_2(y))$, respectively. After coloring these vertices, there are at least $\Delta$ available colors for the coloring $\{v, (v,u), (u,v), u\}$. Since $\Delta\geq4$, we can color properly the remain vertices.\\
		(iii) $G$ has a 2-vertex $v$ with adjacent neighbors $u$ and $w$. Consider $G'=G-v$. Again, $G'$ is outerplanar  with smaller order and maximum degree at most $\Delta$ and so has a $vi$-simultaneous $(\Delta+3,2)$-coloring named $c$. Suppose $(u,w)$ is colored by color $r$ and $(w,u)$ by color $s$. We now assign color $r$ to $(v,w)$ and color $s$ to $(v,u)$. This does not produce any conflict. The degree of $u$ and $w$ in $G'$ is at most $\Delta-1$, so similar to the previous cases, there is at least one color left to color each of the $i$-vertices $(u,v)$ and $(w,v)$. After coloring these vertices, there are at least $\Delta-3$ available colors for the coloring $v$. Since $\Delta\geq4$, we can complete the coloring.
}\end{proof}
\begin{corollary}\label{vi-outplanar}
	$\chi_{vi}(G)\leq \Delta(G)+3$ for any outerplanar graph $G$.
\end{corollary}
Due to the previous corollary, the Conjecture~\ref{cmahsa} is true for all outerplanar graphs. Also, Corollary \ref{vi-outplanar} implies that the $vi$-simultanious chromatic number of any outerplanar graph $G$ is equal to either $\Delta(G)+2$ or  $\Delta(G)+3$. Those outerplanar graphs $G$ for which $\chi_{vi}(G)=\Delta(G)+2$ are said to belong to $vi$-class one, and the others to $vi$-class two. The $vi$-class number of graph is defined as follows.
\begin{definition}
	Let $G$ be a nonempty graph with $vi$-simultanious chromatic number equal to $\Delta(G)+1+s$, where $s\in\mathbb{N}$. We say that $G$ is a graph of $vi$-class $s$.
\end{definition}
For example, Theorem \ref{complete} shows that all complete graphs are of $vi$-class 2 and Theorem \ref{tree}, shows that any forest with maximum degree at least 2 is a graph of $vi$-class 1.\\
In \cite{3power3subdivision}, the authors posed this question: Is there any graph $G$ with $\Delta(G)=3$ such that $\chi(G^{\frac{3}{3}})=6$? In fact, the question is about the existence of subcubic graphs of $vi$-class 2 or 3. In the next theorem, we show that the upper bound of Theorems \ref{delta3} and \ref{vi-outplanar} are tight and there are infinite subcubic graphs of $vi$-class two (with $vi$-simultaneous chromatic number 6).
\begin{theorem}\label{cubic6colors}
	Any graph that contains at least one of the following subgraphs (Figure \ref{subgraphs}) is not $vi$-simultaneous 5-colorable. Especially, for any outerplanar cubic graph $G$ that contains at least one of these subgraphs, $\chi_{vi}(G)=6$.
\end{theorem}
\begin{figure}[h]
	\begin{center}
		\begin{tikzpicture}[scale=0.8]
		\tikzset{vertex/.style = {shape=circle,draw, fill=black, line width=1pt,opacity=1.0, inner sep=2pt}}
		\tikzset{edge/.style = {-,> = latex', line width=1pt,opacity=1.0}}
		\node[vertex] (a1) at (0,0) {};
		\node () at  (-0.4,0) {$a$};
		\node[vertex] (a2) at (2,0) {};
		\node () at  (2.4,0) {$b$};
		\node[vertex] (a3) at  (1,1.4) {};
		\node () at  (1,1.8) {$c$};
		\node[vertex] (a4) at  (1,-1.4) {};
		\node () at  (1,-1.8) {$d$};
		\node () at  (1,-2.8) {$G_1$};
		\node[vertex] (a5) at  (3,-1.4) {};
		\node () at  (2.6,-1.4) {$a$};
		\node[vertex] (a6) at  (5,-1.4) {};
		\node () at  (5.4,-1.4) {$b$};
		\node[vertex] (a7) at (4,0) {};
		\node () at  (3.6,0) {$c$};
		\node[vertex] (a8) at (4,1.4) {};
		\node () at  (4,1.8) {$d$};
		\node[vertex] (a9) at (3,-2.8) {};
		\node () at  (2.6,-2.8) {$e$};
		\node[vertex] (a10) at (5,-2.8) {};
		\node () at  (5.4,-2.8) {$f$};
		\node () at  (4,-2.8) {$G_2$};
		\node[vertex] (a11) at (6,0) {};
		\node () at  (5.6,0) {$a$};
		\node[vertex] (a12) at (8,0) {};
		\node () at  (8.4,0) {$b$};
		\node[vertex] (a13) at (7,1.4) {};
		\node () at  (7,1.8) {$c$};
		\node[vertex] (a14) at (6,-1.4) {};
		\node () at  (6,-1.8) {$d$};
		\node[vertex] (a15) at (8,-1.4) {};
		\node () at  (7.6,-1.8) {$e$};
		\node[vertex] (a16) at (8,-2.8) {};
		\node () at  (7.6,-2.8) {$f$};
		\node () at  (6.8,-2.8) {$G_3$};
		\node[vertex] (a17) at (9,0) {};
		\node () at  (9.4,-0.4) {$a$};
		\node[vertex] (a18) at (11,0) {};
		\node () at  (10.6,-0.4) {$b$};
		\node[vertex] (a19) at (10,1.4) {};
		\node () at  (10,1.8) {$c$};
		\node[vertex] (a20) at (9,-1.4) {};
		\node () at  (8.6,-1.4) {$d$};
		\node[vertex] (a21) at (10,-1.4) {};
		\node () at  (10,-1.8) {$f$};
		\node[vertex] (a22) at (11,-1.4) {};
		\node () at  (11.4,-1.4) {$e$};
		\node[vertex] (a23) at (9,-2.8) {};
		\node () at  (8.6,-2.8) {$h$};
		\node[vertex] (a24) at (11,-2.8) {};
		\node () at  (11.4,-2.8) {$k$};
		\node () at  (10,-2.8) {$G_4$};
		\draw[edge] (a1) -- (a2) -- (a3) -- (a1) -- (a4) -- (a2);
		\draw[edge] (a8) -- (a7) -- (a5) -- (a6) -- (a7);
		\draw[edge] (a5) -- (a9);
		\draw[edge] (a6) -- (a10);
		\draw[edge] (a15) -- (a14) -- (a11) -- (a12) -- (a15) -- (a16);
		\draw[edge] (a11) -- (a13) -- (a12);
		\draw[edge] (a24) -- (a22) -- (a21) -- (a20) -- (a23);
		\draw[edge] (a22) -- (a18) -- (a17) -- (a20);
		\draw[edge] (a17) -- (a19) -- (a18);
		\end{tikzpicture}
		\caption{Four subgraphs of Theorem \ref{cubic6colors}.}
		\label{subgraphs}
	\end{center}
\end{figure}
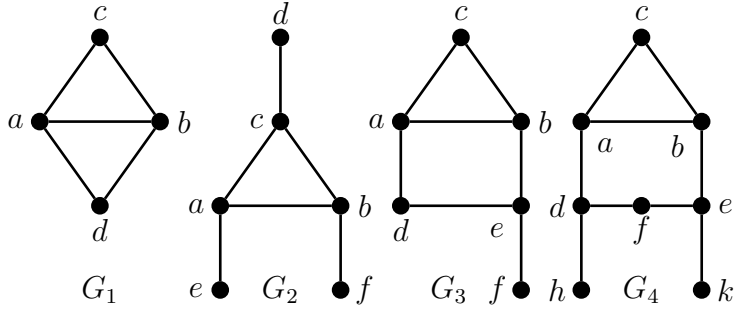
\begin{proof}{
		Let $s$ be a $vi$-simultanious $5$-coloring of $G_i$ ($1\leq i\leq4$). We have following cases.
		\begin{itemize}
			\item 
		($G_1$)  Theorem \ref{firstlem} implies that $s((b,a))=s((c,a))=s((d,a))$ and $s((a,b))=s((c,b))=s((d,b))$. Each color class of $s$ has at most 3 vertices. In addition, each color class that contains the color $s(a)$ or $s(b)$ has at most 2 vertices. Therefore, 5 colors are not enough and 	$\chi_{vi}(G)\geq 6$.\\
		\item
		($G_2$) Theorem \ref{firstlem} implies that $s((b,a))=s((c,a))=s((e,a))$, $s((a,b))=s((c,b))=s((f,b))$ and $s((a,c))=s((b,c))=s((d,c))$. In addition,
		\[|\{s(a),s(b),s(c)\}|=|\{s((a,b)),s((b,a)),s((a,c))\}|=3\]
		and $\{s(a),s(b),s(c)\}\cap\{s((a,b)),s((b,a)),s((a,c))\}=\varnothing$. So 	$\chi_{vi}(G)\geq 6$.\\
		\item
		($G_3$) Without loss of generality, suppose that $s(b)=1$, $s((b,c))=2$, $s((b,a))=3$ and $s((b,e))=4$. Lemma \ref{firstlem} implies that $s((c,b))=s((a,b))=s((e,b))$, $s((c,a))=s((b,a))=s((d,a))=3$ and $s((b,e))=s((d,e))=s((f,e))=4$. Suppose that $s((c,b))=s((a,b))=s((e,b))=5$. Since $c$ is a $vi$-simultanious $5$-coloring, $s(c)=4$. Consecutively we have $s(a)=2$ and then $s((a,c))=1$. Now we have no choice for the color of $(a,d)$, a contradiction. So 	$\chi_{vi}(G)\geq 6$.\\
		\item
		($G_4$) Without loss of generality, suppose that $s(b)=1$, $s((b,c))=2$, $s((b,a))=3$ and $s((b,e))=4$. Lemma \ref{firstlem} implies that $s((c,b))=s((a,b))=s((e,b))$, $s((c,a))=s((b,a))=s((d,a))=3$, $s((b,e))=s((f,e))=s((k,e))=4$ and $s((a,d))=s((f,d))=s((h,d))$. Suppose that $s((c,b))=s((a,b))=s((e,b))=5$.  Since $c$ is a $vi$-simultanious $5$-coloring, $s(c)=4$. Consecutively we have $s(a)=2$ and then $s((a,c))=1$. Since $s((f,d))\neq s((f,e))=4$, $s((a,d))\neq 4$. Now we have no choice for the color of $(a,d)$, a contradiction. So 	$\chi_{vi}(G)\geq 6$.
				\end{itemize}
}\end{proof}
\subsection{Proof of the Theorem~\ref{2}}
 To prove of the Theorem~\ref{2}, first we are going to prove the first case in the Theorem \ref{delta3girth4}. After that, in order to prove the two other cases, given the fact that   $\chi_{vi}(K_2)=4$ and $\chi_{vi}(G)\geq \Delta+2$ and by  Corollary~\ref{blocks}, we are just going to prove them for 2-connected outerplanar graphs with the desired properties in Theorems~\ref{delta4girth6} and \ref{delta5girth4}.
\begin{theorem}\label{delta3girth4}
	If $G$ is a 2-connected outerplanar graph with $\Delta(G)=3$ and $g(G)\geq4$, then $\chi_{vi,1}(G)\leq6$.
\end{theorem}
\begin{proof}{
Suppose that $\Delta(G)=\Delta$ and $g(G)=g$. We prove the theorem by induction on the order of $G$. Since $G$ is 2-connected, $\Delta\geq3$ and $g\geq4$, $G$ has at least six vertices. In Figure~\ref{figdelta3girth4}, the only 2-connected outerplanar graph of order $6$ with $\Delta=3$ and $g=4$ with a $vi$-simultanious $(6,1)$-coloring is shown.
		\begin{figure}[h]
			\begin{center}
				\begin{tikzpicture}[scale=0.8]
				\tikzset{vertex1/.style = {shape=circle,draw, fill=black, line width=1pt,opacity=1.0, inner sep=2pt}}
				\tikzset{vertex/.style = {shape=circle,draw, fill=white, line width=1pt,opacity=1.0, inner sep=2pt}}
				\tikzset{edge/.style = {-,> = latex', line width=1pt,opacity=1.0}}
				\node[vertex1] (a1) at (0,0) {};
				\node () at  (0,0.4) {$1$};
				\node[vertex] (a2) at (1,0) {};
				\node () at  (1,0.4) {$2$};
				\node[vertex] (a3) at  (2,0) {};
				\node () at  (2,0.4) {$5$};
				\node[vertex1] (a4) at  (3,0) {};
				\node () at  (3,0.4) {$3$};
				\node[vertex] (a5) at  (3,-1) {};
				\node () at  (3.4,-1) {$4$};
				\node[vertex] (a6) at  (3,-2) {};
				\node () at  (3.4,-2) {$2$};
				\node[vertex1] (a7) at (3,-3) {};
				\node () at  (3.4,-3) {$1$};
				\node[vertex] (a8) at (2,-3) {};
				\node () at  (2,-3.4) {$3$};
				\node[vertex] (a9) at (1,-3) {};
				\node () at  (1,-3.4) {$4$};
				\node[vertex1] (a10) at (0,-3) {};
				\node () at  (0,-3.4) {$2$};
				\node[vertex] (a11) at (-1,-3) {};
				\node () at  (-1,-3.4) {$1$};
				\node[vertex] (a12) at (-2,-3) {};
				\node () at  (-2,-3.4) {$3$};
				\node[vertex1] (a13) at (-3,-3) {};
				\node () at  (-3,-3.4) {$5$};
				\node[vertex] (a14) at (-3,-2) {};
				\node () at  (-3.4,-2) {$4$};
				\node[vertex] (a15) at (-3,-1) {};
				\node () at  (-3.4,-1) {$1$};
				\node[vertex1] (a16) at (-3,0) {};
				\node () at  (-3.4,0) {$2$};
				\node[vertex] (a17) at (-2,0) {};
				\node () at  (-2,0.4) {$5$};
				\node[vertex] (a18) at (-1,0) {};
				\node () at  (-1,0.4) {$4$};
				\node[vertex] (a19) at (0,-1) {};
				\node () at  (-0.4,-1) {$3$};
				\node[vertex] (a20) at (0,-2) {};
				\node () at  (-0.4,-2) {$5$};
				\draw[edge] (a1) -- (a2) -- (a3) -- (a4) -- (a5) -- (a6) -- (a7) -- (a8) -- (a9) -- (a10) -- (a11) -- (a12) -- (a13) -- (a14) -- (a15) -- (a16) -- (a17) -- (a18) -- (a1) -- (a19) -- (a20) -- (a10);
				\end{tikzpicture}
				\caption{$2$-connected outerplanar graph $G$ of order $6$ with $\Delta(G)=3$ and  $g(G)=4$.}
				\label{figdelta3girth4}
			\end{center}
		\end{figure}
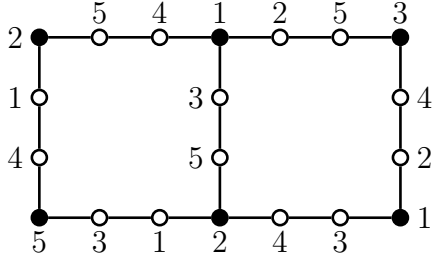
		Now suppose that $G$ is a $2$-connected outerplanar graph of order $n$ with $\Delta=3$ and $g=4$, and the statement is true for all $2$-connected outerplanar graphs of order less than $n$ with $\Delta=3$ and $g=4$. Since $G$ is $2$-connected, $G$ has an end face $f=[v_iv_{i+1} \ldots v_j]$ of degree at least $4$. Let $S$ be the set of 2-vertices of $f$ and $G'=G-S$.\\
		First, suppose that $f=[v_iv_{i+1}v_{i+2}v_{i+3}]$. If $\Delta(G')=2$, then  by Lemma~\ref{cycles}, $G'$ has a $vi$-simultaneous $(5,1)$-coloring $c$ and if $\Delta(G')=3$, then by  induction hypothesis $G'$ has a $vi$-simultaneous $(6,1)$-coloring such as $c$. Now, it suffices to extend $c$ to a proper $vi$-simultaneous $(6,1)$-coloring for $G$.\\
		Since $v_i$, $v_{i+3}$, $(v_i,v_{i+3})$ and $(v_{i+3},v_i)$ are pairwise adjacent in $(G')^{\frac{3}{3}}$, they must have different colors. Without loss of generality, suppose that $c(v_i)=1$, $c((v_i,v_{i+3}))=2$, $c((v_{i+3},v_i))=3$ and $c(v_{i+3})=4$. Note that, by coloring of $(G')^{\frac{3}{3}}$, we have $c((v_{i-1},v_i))=c((v_{i+3},v_i))$ and $c((v_{i+4},v_{i+3}))=c((v_{i},v_{i+3}))$. Now, color two $i$-vertices $(v_{i+1},v_i)$ and $(v_{i+2},v_{i+3})$ with colors $3$ and $2$, respectively. It can be easily seen that there are at least two available colors for each of $i$-vertices $(v_{i},v_{i+1})$ and $(v_{i+3},v_{i+2})$. First, assign different colors to these two $i$-vertices. Also, color $(v_{i+2},v_{i+1})$ and $(v_{i+1},v_{i+2})$ as same as $i$-vertices $(v_{i},v_{i+1})$ and $(v_{i+3},v_{i+2})$, respectively. 
		By coloring two $t$-vertices $v_{i+1}$ and  $v_{i+2}$ with different available colors, we have a proper $vi$-simultaneous $(6,1)$-coloring for $G$.\\
		Now, suppose that $f=[v_iv_{i+1}v_{i+2}\ldots v_{j}]$, $j\geq i+4$ and consider the graph $G'=(G-v_{i+1})+e$ where $e=\{v_{i},v_{i+2}\}$. Color two $i$-vertices $(v_i,v_{i+1})$ and $(v_{i+1},v_{i})$ with colors $c((v_i,v_{i+2}))$ and $c((v_{i+2},v_{i}))$, respectively. Also, assign color $c((v_i,v_{i+2}))$ to the $i$-vertex $(v_{i+2},v_{i+1})$. Now, we are going to recolor vertices $(v_{i+3},v_{i+2})$ and $v_{i+2}$ and color the $t$-vertex $v_{i+1}$ as well.  It can be easily seen that, there is one available color for the $i$-vertex $(v_{i+3},v_{i+2})$ and $3$ available colors for the vertices $v_{i+1}$ and $v_{i+2}$. By coloring these vertices with different colors, we have a proper $vi$-simultaneous $(6,1)$-coloring of $G$.
}\end{proof}
\begin{theorem}\label{delta4girth4}
	If $G$ is a $2$-connected outerplanar graph with $\Delta(G)\geq4$ and $g(G)\geq4$, then $\chi_{vi,1}(G)\leq\Delta+3$.
\end{theorem}
\begin{proof}{
We prove the theorem by induction on the number of vertices of $G$. Since $G$ is 2-connected, $\Delta\geq4$ and $g\geq4$, $G$ has at least eight vertices.  In Figure~\ref{d4g4},  a $vi$-simultanious $(6,1)$-coloring of the only $2$-connected outerplanar graph of order $8$ with maximum degree $4$ and girth at least $4$ is shown.

 Let $G$ be a $2$-connected outerplanar graph of order $n$ and $\Delta(G)=\Delta$. Since $\Delta\geq 4$ and $g(G)\geq 4$, $G$ has an end face $f=[v_iv_{i+1}\ldots v_{j-1}v_j]$ of degree at least $4$. \\
		Suppose that the theorem is true for all $2$-connected outerplanar graphs with maximum degree at least $4$, girth at least $4$ and less than $n$ vertices.
		\begin{figure}[h]
			\begin{center}
				\begin{tikzpicture}[scale=0.7]
				\tikzset{vertex1/.style = {shape=circle,draw, fill=black, line width=1pt,opacity=1.0, inner sep=2pt}}
				\tikzset{vertex/.style = {shape=circle,draw, fill=white, line width=1pt,opacity=1.0, inner sep=2pt}}
				\tikzset{edge/.style = {-,> = latex', line width=1pt,opacity=1.0}}
				\node[vertex1] (a1) at (0,0) {};
				\node () at  (0.4,-0.4) {$1$};
				\node[vertex] (a2) at (1,0) {};
				\node () at  (1,0.4) {$3$};
				\node[vertex] (a3) at  (2,0) {};
				\node () at  (2,0.4) {$6$};
				\node[vertex1] (a4) at  (3,0) {};
				\node () at  (3,0.4) {$2$};
				\node[vertex] (a5) at  (3,-1) {};
				\node () at  (3.4,-1) {$1$};
				\node[vertex] (a6) at  (3,-2) {};
				\node () at  (3.4,-2) {$3$};
				\node[vertex1] (a7) at (3,-3) {};
				\node () at  (3.4,-3) {$6$};
				\node[vertex] (a8) at (2,-3) {};
				\node () at  (2,-3.4) {$4$};
				\node[vertex] (a9) at (1,-3) {};
				\node () at  (1,-3.4) {$1$};
				\node[vertex1] (a10) at (0,-3) {};
				\node () at  (0,-3.4) {$2$};
				\node[vertex] (a11) at (-1,-3) {};
				\node () at  (-1,-3.4) {$3$};
				\node[vertex] (a12) at (-2,-3) {};
				\node () at  (-2,-3.4) {$4$};
				\node[vertex1] (a13) at (-3,-3) {};
				\node () at  (-3,-3.4) {$1$};
				\node[vertex] (a14) at (-3,-2) {};
				\node () at  (-3.4,-2) {$5$};
				\node[vertex] (a15) at (-3,-1) {};
				\node () at  (-3.4,-1) {$3$};
				\node[vertex1] (a16) at (-3,0) {};
				\node () at  (-3.4,0) {$2$};
				\node[vertex] (a17) at (-2,0) {};
				\node () at  (-2,0.4) {$6$};
				\node[vertex] (a18) at (-1,0) {};
				\node () at  (-1,0.4) {$5$};
				\node[vertex] (a19) at (0,-1) {};
				\node () at  (-0.4,-1) {$4$};
				\node[vertex] (a20) at (0,-2) {};
				\node () at  (-0.4,-2) {$6$};
				\node[vertex] (a21) at (-3,1) {};
				\node () at  (-3.4,1) {$4$};
				\node[vertex] (a22) at (-3,2) {};
				\node () at  (-3.4,2) {$5$};
				\node[vertex1] (a23) at (-3,3) {};
				\node () at  (-3.4,3) {$3$};
				\node[vertex] (a24) at (-2,3) {};
				\node () at  (-2,3.4) {$2$};
				\node[vertex] (a25) at (-1,3) {};
				\node () at  (-1,3.4) {$4$};
				\node[vertex1] (a26) at (0,3) {};
				\node () at  (0,3.4) {$5$};
				\node[vertex] (a27) at (0,2) {};
				\node () at  (0.4,2) {$6$};
				\node[vertex] (a28) at (0,1) {};
				\node () at  (0.4,1) {$2$};
				\draw[edge] (a1) -- (a2) -- (a3) -- (a4) -- (a5) -- (a6) -- (a7) -- (a8) -- (a9) -- (a10) -- (a11) -- (a12) -- (a13) -- (a14) -- (a15) -- (a16) -- (a17) -- (a18) -- (a1) -- (a19) -- (a20) -- (a10);
				\draw[edge] (a16) -- (a21) -- (a22) -- (a23) -- (a24) -- (a25) -- (a26) -- (a27) -- (a28) -- (a1);
				\end{tikzpicture}
				\caption{A $2$-connected outerplanar graph  of order 8 with maximum degree 4 and girth 4.}
				\label{d4g4}
			\end{center}
		\end{figure}
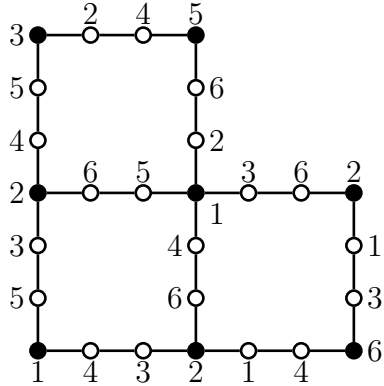
		Consider the end face $f=[v_iv_{i+1}\ldots v_{j-1}v_j]$ of degree at least $4$ and  let $H$ be the induced subgraph of $G$ on $t$-vertices of degree $2$ of  $f$. If $\Delta(G\setminus H)\leq 3$, then by Theorem~\ref{delta3girth4}, $(G\setminus H)^{\frac{3}{3}}$ has a $6$-proper coloring, named $c$ in which for any vertex $u\in V(G\setminus H)$, $|\{c(v): v\in I_2(u) \}|=1$. Also, if $\Delta(G\setminus H)\geq 4$, then by induction hypothesis, $c$ is a $(\Delta+3)$-proper coloring for $(G\setminus H)^{\frac{3}{3}}$ with desired property. It is enough to extend $c$ to a $(\Delta+3)$-proper coloring of $G^{\frac{3}{3}}$ with desired property of the theorem.\\
		First suppose that $f=[v_iv_{i+1}v_{i+2}v_{i+3}]$. Consider the graph $G\setminus H$ and its desired coloring $c$ with at most $(\Delta+3)$ colors. First, color two $i$-vertices $(v_{i+1},v_i)$ and $(v_{i+2},v_{i+3})$ as same as $i$-vertices $(v_{i},v_{i+3})$ and $(v_{i+3},v_{i})$, respectively. Now color two $i$-vertices $(v_{i},v_{i+1})$ and $(v_{i+3},v_{i+2})$ with different available colors. Also, color two $i$-vertices $(v_{i+2},v_{i+1})$ and $(v_{i+1},v_{i+2})$ as same as $i$-vertices $(v_{i},v_{i+1})$ and $(v_{i+3},v_{i+3})$, respectively.
		Now it can be easily seen that each of $t$-vertices $v_{i+1}$ and $v_{i+2}$ have at least $\Delta-1\geq 3$ available colors. By coloring those $t$-vertices we have a $(\Delta+3)$-coloring for $G^{\frac{3}{3}}$ with the desired property. \\
		Now, suppose that $f=[v_iv_{i+1}\ldots v_{j-1}v_j]$ is an end face of degree at least $5$ with $j\geq i+4$ and consider the graph $G'=(G-v_{i+1})+e$ where $e=\{v_{i}v_{i+2}\}$. Color two $i$-vertices $(v_i,v_{i+1})$ and $(v_{i+1},v_{i})$ with colors $c((v_i,v_{i+2}))$ and $c((v_{i+2},v_{i}))$, respectively. Also, assign color $c((v_i,v_{i+2}))$ to the $i$-vertex $(v_{i+2},v_{i+1})$. Now, we are going to recolor vertices $(v_{i+3},v_{i+2})$ and $v_{i+2}$ and color the $t$-vertex $v_{i+1}$ as well.  It can be easily seen that, there are two available colors for the $i$-vertex $(v_{i+3},v_{i+2})$ and $\Delta-1\geq 3$ available colors for each of the $t$-vertices $v_{i+1}$ and $v_{i+2}$. By coloring these vertices with different colors, we have a $(\Delta+3)$-coloring for $(G)^{\frac{3}{3}}$ with desired property of the theorem.
}\end{proof}
According to Theorems \ref{cycles}, \ref{delta3girth4}, \ref{delta4girth4}, Theorem \ref{tree} and Lemma \ref{thirdlem}, we conclude the following result. A cycle $C_3$ is often called a triangle and a triangle-free graph is a graph containing no triangle.
\begin{theorem}
	$\chi_{vi,1}(G)\leq\Delta(G)+3$ for any  triangle-free outerplanar graph.
\end{theorem}
\begin{theorem}\label{delta4girth6}
	If $G$ is a $2$-connected outerplanar graph with $\Delta(G)=\Delta\geq4$ and  $g(G)=g\geq6$, then $\chi_{vi,1}(G)=\Delta(G)+2$.
\end{theorem}
\begin{proof}{
We prove the theorem by induction on the number of vertices of $G$. Since $G$ is 2-connected, $\Delta\geq4$ and $g\geq6$, $G$ has at least $14$ vertices.  In Figure~\ref{fignew},  a $vi$-simultanious $(6,1)$-coloring of the only $2$-connected outerplanar graph of order $14$ with maximum degree $4$ and girth at least $6$ is shown.\\
Let $|V(G)|=n$ and suppose that the theorem is true for all $2$-connected outerplanar graphs with maximum degree at least $4$, girth at least $6$ and less than $n$ vertices. Since $\chi_{vi}(G)\geq\omega(G^{\frac{3}{3}})=\Delta+2$, it suffices to present a $(\Delta+2)$-proper coloring for $G^{\frac{3}{3}}$. Since $\Delta\geq 4$ and $g(G)\geq 6$, $G$ has an end face $f=[v_iv_{i+1}\ldots v_{j-1}v_j]$ of degree at least $6$.\\
		\begin{figure}[h]
			\begin{center}
				\resizebox{8cm}{6cm}{%
					\begin{tikzpicture}[scale=0.7]
					\tikzset{vertex1/.style = {shape=circle,draw, fill=black, line width=1pt,opacity=1.0, inner sep=2pt}}
					\tikzset{vertex/.style = {shape=circle,draw, fill=white, line width=1pt,opacity=1.0, inner sep=2pt}}
					\tikzset{edge/.style = {-,> = latex', line width=1pt,opacity=1.0}}
					\node[vertex1] (1) at (1,1) {};
					\node () at  (1.4,1.4) {$3$};
					\node[vertex] (2) at (2,1) {};
					\node () at  (2,1.4) {$4$};
					\node[vertex] (3) at  (3,1) {};
					\node () at  (3,1.4) {$1$};
					\node[vertex1] (4) at  (4,1) {};
					\node () at  (4,1.4) {$5$};
					\node[vertex] (5) at  (5,1) {};
					\node () at  (5,1.4) {$3$};
					\node[vertex] (6) at  (6,1) {};
					\node () at  (6,1.4) {$4$};
					\node[vertex1] (7) at (7,1) {};
					\node () at  (7.4,1.4) {$2$};
					\node[vertex] (8) at (8,1) {};
					\node () at  (8,1.4) {$1$};
					\node[vertex] (9) at (9,1) {};
					\node () at  (9,1.4) {$3$};
					\node[vertex1] (10) at (10,1) {};
					\node () at  (10,1.4) {$6$};
					\node[vertex] (11) at (11,1) {};
					\node () at  (11,1.4) {$5$};
					\node[vertex] (12) at (12,1) {};
					\node () at  (12,1.4) {$1$};
					\node[vertex1] (13) at (13,1) {};
					\node () at  (12.6,1.4) {$2$};
					\node[vertex] (14) at (13,2) {};
					\node () at  (12.6,2) {$3$};
					\node[vertex] (15) at (13,3) {};
					\node () at  (12.6,3) {$5$};
					\node[vertex1] (16) at (13,4) {};
					\node () at  (13,4.4) {$4$};
					\node[vertex] (17) at (12,4) {};
					\node () at  (12,4.4) {$2$};
					\node[vertex] (18) at (11,4) {};
					\node () at  (11,4.4) {$3$};
					\node[vertex1] (19) at (10,4) {};
					\node () at  (10,4.4) {$5$};
					\node[vertex] (20) at (9,4) {};
					\node () at  (9,4.4) {$6$};
					\node[vertex] (21) at (8,4) {};
					\node () at  (8,4.4) {$2$};
					\node[vertex1] (22) at (7,4) {};
					\node () at  (7.4,3.6) {$1$};
					\node[vertex] (23) at (7,5) {};
					\node () at  (7.4,5) {$5$};
					\node[vertex] (24) at (7,6) {};
					\node () at  (7.4,6) {$6$};
					\node[vertex1] (25) at (7,7) {};
					\node () at  (7.4,7) {$4$};
					\node[vertex] (26) at (7,8) {};
					\node () at  (7.4,8) {$2$};
					\node[vertex] (27) at (7,9) {};
					\node () at  (7.4,9) {$5$};
					\node[vertex1] (28) at (7,10) {};
					\node () at  (7.4,10) {$6$};
					\node[vertex] (29) at (6,10) {};
					\node () at  (6,9.6) {$1$};
					\node[vertex] (30) at (5,10) {};
					\node () at  (5,9.6) {$2$};
					\node[vertex1] (31) at (4,10) {};
					\node () at  (3.6,10) {$3$};
					\node[vertex] (32) at (4,9) {};
					\node () at  (3.6,9) {$5$};
					\node[vertex] (33) at (4,8) {};
					\node () at  (3.6,8) {$1$};
					\node[vertex1] (34) at (4,7) {};
					\node () at  (3.6,7) {$2$};
					\node[vertex] (35) at (4,6) {};
					\node () at  (3.6,6) {$4$};
					\node[vertex] (36) at (4,5) {};
					\node () at  (3.6,5) {$5$};
					\node[vertex1] (37) at (4,4) {};
					\node () at  (4,3.6) {$3$};
					\node[vertex] (38) at (3,4) {};
					\node () at  (3,4.4) {$2$};
					\node[vertex] (39) at (2,4) {};
					\node () at  (2,4.4) {$4$};
					\node[vertex1] (40) at (1,4) {};
					\node () at  (1,4.4) {$5$};
					\node[vertex] (41) at (1,3) {};
					\node () at  (1.4,3) {$1$};
					\node[vertex] (42) at (1,2) {};
					\node () at  (1.4,2) {$2$};
					\node[vertex] (43) at (5,4) {};
					\node () at  (5,4.4) {$6$};
					\node[vertex] (44) at (6,4) {};
					\node () at  (6,4.4) {$4$};
					\node[vertex] (45) at (7,3) {};
					\node () at  (6.6,3) {$3$};
					\node[vertex] (46) at (7,2) {};
					\node () at  (6.6,2) {$6$};
					\draw[edge] (1) -- (2) -- (3) -- (4) -- (5) -- (6) -- (7) -- (8) -- (9) -- (10) -- (11) -- (12) -- (13) -- (14) -- (15) -- (16) -- (17) -- (18) -- (19) -- (20) -- (21) -- (22) -- (23) -- (24) -- (25) -- (26) -- (27) -- (28) -- (29) -- (30) -- (31) -- (32) -- (33) -- (34) -- (35) -- (36) -- (37) -- (38) -- (39) -- (40) -- (41) -- (42) -- (1);
					\draw[edge] (37) -- (43) -- (44) -- (22) -- (45) -- (46) -- (7);
					\end{tikzpicture}}
				\caption{A $2$-connected outerplanar graph  of order 14 with maximum degree 4 and girth 6.}
				\label{fignew}
			\end{center}
		\end{figure}
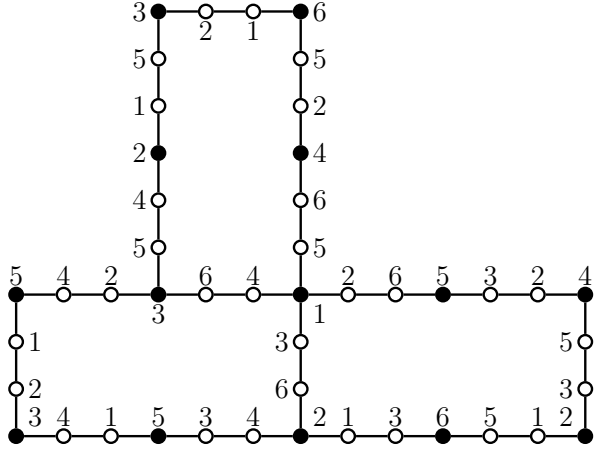
Consider the end face $f=[v_iv_{i+1}\ldots v_{j-1}v_j]$ of degree at least $6$ and  let $H$ be the induced subgraph of $G$ on $t$-vertices of degree $2$ of  $f$. If $\Delta(G\setminus H)\leq 3$, then by Theorem~\ref{delta3girth4}, $(G\setminus H)^{\frac{3}{3}}$ has a $6$-proper coloring, named $c$ in which for any vertex $u\in V(G\setminus H)$, $|c(I_2(u))|=1$. Also, if $\Delta(G\setminus H)\geq 4$, then by induction hypothesis, $c$ is a $(\Delta+2)$-proper coloring for $(G\setminus H)^{\frac{3}{3}}$ with desired property. It is enough to extend $c$ to a $(\Delta+2)$-proper coloring of $G^{\frac{3}{3}}$ with desired property of the theorem.\\
First suppose that $f=[v_iv_{i+1}v_{i+2}v_{i+3}v_{i+4}v_{i+5}]$. Consider the graph $G\setminus H$ with its $(\Delta+2)$-desired coloring $c$. First, color two $i$-vertices $(v_{i+1},v_i)$ and $(v_{i+4},v_{i+5})$ as same as $i$-vertices $(v_{i},v_{i+5})$ and $(v_{i+5},v_{i})$, respectively. Now color two $i$-vertices $(v_{i},v_{i+1})$ and $(v_{i+5},v_{i+4})$ with different available colors. Also, color two $i$-vertices $(v_{i+2},v_{i+1})$ and $(v_{i+3},v_{i+4})$ as same as $i$-vertices $(v_{i},v_{i+1})$ and $(v_{i+5},v_{i+4})$, respectively.
Now, each of $i$-vertices $(v_{i+1},v_{i+2})$ and $(v_{i+4},v_{i+3})$ has at least $\Delta-1\geq 3$ available colors. Color them with different colors. Also, color $(v_{i+2},v_{i+3})$ and $(v_{i+3},v_{i+2})$ like $(v_{i+1},v_{i+2})$ and $(v_{i+3},v_{i+4})$.
Now it can be easily seen that each of $t$-vertices $v_{i+1}$ and $v_{i+4}$ has at least $\Delta-2\geq 2$ available colors and $t$-vertices $v_{i+2}$ and $v_{i+3}$ have at least $\Delta-1\geq 3$ available colors. By coloring those $t$-vertices, we have a $(\Delta+2)$-coloring for $G^{\frac{3}{3}}$ with the desired property.\\	
Now, suppose that $f=[v_iv_{i+1}\ldots v_{j-1}v_j]$ is an end face of degree at least $7$ with $j\geq i+6$ and consider the graph $G'=(G-v_{i+1})+e$ where $e=\{v_{i}v_{i+2}\}$. Color two $i$-vertices $(v_i,v_{i+1})$ and $(v_{i+1},v_{i})$ with colors $c((v_i,v_{i+2}))$ and $c((v_{i+2},v_{i}))$, respectively. Also, assign color $c((v_i,v_{i+2}))$ to the $i$-vertex $(v_{i+2},v_{i+1})$. Now, we are going to recolor vertices $(v_{i+3},v_{i+2})$ and $v_{i+2}$ and color the $t$-vertex $v_{i+1}$ as well.  It can be easily seen that, there is one available color for the $i$-vertex $(v_{i+3},v_{i+2})$ and $\Delta-1\geq 3$ available colors for $t$-vertices $v_{i+1}$ and $v_{i+2}$. By coloring these vertices with different colors, we have a $(\Delta+2)$-coloring for $G^{\frac{3}{3}}$ which completes the proof.
}\end{proof}
\begin{theorem} \em{\cite{injective}}\label{injective}
	If $G$ is a $2$-connected outerplanar graph, then $G$ has an end face $f =[v_iv{i+1}\ldots v_j ]$, where either $deg(v_i) < 5$ or $deg(v_j) < 5$.
\end{theorem}
\begin{theorem}\label{delta5girth4}
	If $G$ is a $2$-connected outerplanar graph of order $n$ with $\Delta(G)=\Delta\geq5$ and $g(G)=g\geq4$, then $\chi_{vi,1}(G)=\Delta(G)+2$.
\end{theorem}
\begin{proof}{
We prove the theorem by induction on the order of $G$. Since $G$ is 2-connected, $\Delta\geq5$ and $g\geq4$, $G$ has at least $10$ vertices.  In Figure~\ref{d5g4}, a $vi$-simultanious $(7,1)$-coloring of the only $2$-connected outerplanar graph of order $10$ with maximum degree $5$ and girth at least $4$ is shown.\\
Suppose that the theorem is true for all $2$-connected outerplanar graphs with maximum degree at least $5$, girth at least $4$ and less than $n$ vertices. By Theorem~\ref{injective}, there is an end face $f =[v_iv_{i+1}\ldots v_j ]$ of $G$ of degree at least $4$, where either $deg(v_i) < 5$ or $deg(v_j) < 5$. Let $H$ be the induced subgraph of $G$ on $t$-vertices of degree $2$ of $f$.\\
		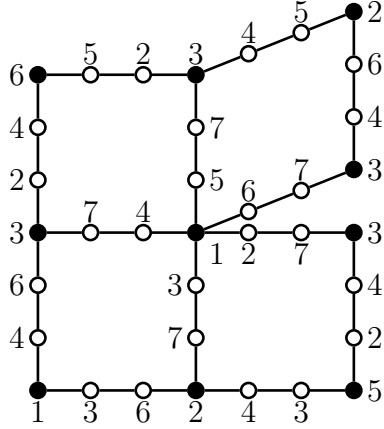
\begin{figure}[h]
			\begin{center}
				\begin{tikzpicture}[scale=0.7]
				\tikzset{vertex1/.style = {shape=circle,draw, fill=black, line width=1pt,opacity=1.0, inner sep=2pt}}
				\tikzset{vertex/.style = {shape=circle,draw, fill=white, line width=1pt,opacity=1.0, inner sep=2pt}}
				\tikzset{edge/.style = {-,> = latex', line width=1pt,opacity=1.0}}
				\node[vertex1] (a1) at (0,0) {};
				\node () at  (0.4,-0.4) {$1$};
				\node[vertex] (a2) at (1,0) {};
				\node () at  (1,-0.4) {$2$};
				\node[vertex] (a3) at  (2,0) {};
				\node () at  (2,-0.4) {$7$};
				\node[vertex1] (a4) at  (3,0) {};
				\node () at  (3.4,0) {$3$};
				\node[vertex] (a5) at  (3,-1) {};
				\node () at  (3.4,-1) {$4$};
				\node[vertex] (a6) at  (3,-2) {};
				\node () at  (3.4,-2) {$2$};
				\node[vertex1] (a7) at (3,-3) {};
				\node () at  (3.4,-3) {$5$};
				\node[vertex] (a8) at (2,-3) {};
				\node () at  (2,-3.4) {$3$};
				\node[vertex] (a9) at (1,-3) {};
				\node () at  (1,-3.4) {$4$};
				\node[vertex1] (a10) at (0,-3) {};
				\node () at  (0,-3.4) {$2$};
				\node[vertex] (a11) at (-1,-3) {};
				\node () at  (-1,-3.4) {$6$};
				\node[vertex] (a12) at (-2,-3) {};
				\node () at  (-2,-3.4) {$3$};
				\node[vertex1] (a13) at (-3,-3) {};
				\node () at  (-3,-3.4) {$1$};
				\node[vertex] (a14) at (-3,-2) {};
				\node () at  (-3.4,-2) {$4$};
				\node[vertex] (a15) at (-3,-1) {};
				\node () at  (-3.4,-1) {$6$};
				\node[vertex1] (a16) at (-3,0) {};
				\node () at  (-3.4,0) {$3$};
				\node[vertex] (a17) at (-2,0) {};
				\node () at  (-2,0.4) {$7$};
				\node[vertex] (a18) at (-1,0) {};
				\node () at  (-1,0.4) {$4$};
				\node[vertex] (a19) at (0,-1) {};
				\node () at  (-0.4,-1) {$3$};
				\node[vertex] (a20) at (0,-2) {};
				\node () at  (-0.4,-2) {$7$};
				\node[vertex] (a21) at (-3,1) {};
				\node () at  (-3.4,1) {$2$};
				\node[vertex] (a22) at (-3,2) {};
				\node () at  (-3.4,2) {$4$};
				\node[vertex1] (a23) at (-3,3) {};
				\node () at  (-3.4,3) {$6$};
				\node[vertex] (a24) at (-2,3) {};
				\node () at  (-2,3.4) {$5$};
				\node[vertex] (a25) at (-1,3) {};
				\node () at  (-1,3.4) {$2$};
				\node[vertex1] (a26) at (0,3) {};
				\node () at  (0,3.4) {$3$};
				\node[vertex] (a27) at (0,2) {};
				\node () at  (0.4,2) {$7$};
				\node[vertex] (a28) at (0,1) {};
				\node () at  (0.4,1) {$5$};
				\node[vertex] (a29) at (1,3.4) {};
				\node () at  (1,3.8) {$4$};
				\node[vertex] (a30) at (2,3.8) {};
				\node () at  (2,4.2) {$5$};
				\node[vertex1] (a31) at (3,4.2) {};
				\node () at  (3.4,4.2) {$2$};
				\node[vertex] (a32) at (3,3.2) {};
				\node () at  (3.4,3.2) {$6$};
				\node[vertex] (a33) at (3,2.2) {};
				\node () at  (3.4,2.2) {$4$};
				\node[vertex1] (a34) at (3,1.2) {};
				\node () at  (3.4,1.2) {$3$};
				\node[vertex] (a35) at (2,0.8) {};
				\node () at  (2,1.2) {$7$};
				\node[vertex] (a36) at (1,0.4) {};
				\node () at  (1,0.8) {$6$};
				\draw[edge] (a1) -- (a2) -- (a3) -- (a4) -- (a5) -- (a6) -- (a7) -- (a8) -- (a9) -- (a10) -- (a11) -- (a12) -- (a13) -- (a14) -- (a15) -- (a16) -- (a17) -- (a18) -- (a1) -- (a19) -- (a20) -- (a10);
				\draw[edge] (a16) -- (a21) -- (a22) -- (a23) -- (a24) -- (a25) -- (a26) -- (a27) -- (a28) -- (a1);
				\draw[edge] (a26) -- (a29) -- (a30) -- (a31) -- (a32) -- (a33) -- (a34) -- (a35) -- (a36) -- (a1);
				\end{tikzpicture}
				\caption{A $2$-connected outerplanar graph of order $10$ with maximum degree $5$ and girth $4$.}\label{d5g4}
			\end{center}
		\end{figure}
If $\Delta(G\setminus H)=4$, then by Theorem~\ref{delta4girth4}, $(G\setminus H)^{\frac{3}{3}}$ has a $7$-proper coloring, named $c$ in which for any vertex $u\in V(G\setminus H)$, $|c(I_2(u))|=1$. Also, if $\Delta\geq 5$, then by induction hypothesis, $c$ is a $(\Delta+2)$-proper coloring for $(G\setminus H)^{\frac{3}{3}}$ with desired property. It is enough to extend $c$ to a $(\Delta+2)$-proper coloring of $G^{\frac{3}{3}}$ with the desired property of the theorem.\\
First suppose that $f=[v_iv_{i+1}v_{i+2}v_{i+3}]$. Consider the graph $G\setminus H$ and its desired coloring $c$ with at most $(\Delta+2)$ colors. First, color two $i$-vertices $(v_{i+1},v_i)$ and $(v_{i+2},v_{i+3})$ as same as $i$-vertices $(v_{i},v_{i+3})$ and $(v_{i+3},v_{i})$, respectively. Now color two $i$-vertices $(v_{i},v_{i+1})$ and $(v_{i+5},v_{i+4})$ with different available colors. Note that, since at least one of the  $t$-vertices $v_i$ or $v_{i+3}$ is of degree at most $4$ in $G$, there are different available colors for these two $t$-vertices. Also, color two $i$-vertices $(v_{i+2},v_{i+1})$ and $(v_{i+1},v_{i+2})$ as same as $i$-vertices $(v_{i},v_{i+1})$ and $(v_{i+3},v_{i+2})$, respectively.
		Now it can be easily seen that, each of $t$-vertices $v_{i+1}$ and $v_{i+2}$ has at least $\Delta-2\geq 3$ available colors. By coloring those $t$-vertices we have a $(\Delta+2)$-coloring for $G^{\frac{3}{3}}$ with the desired property. \\  
		Now, suppose that $f=[v_iv_{i+1}\ldots v_{j-1}v_j]$ is an end face of degree at least $5$ with $j\geq i+4$ and consider the graph $G'=(G-v_{i+1})+e$ where $e=\{v_{i}v_{i+2}\}$. Color two $i$-vertices $(v_i,v_{i+1})$ and $(v_{i+1},v_{i})$ with colors $c((v_i,v_{i+2}))$ and $c((v_{i+2},v_{i}))$, respectively. Also, assign color $c((v_i,v_{i+2}))$ to the $i$-vertex $(v_{i+2},v_{i+1})$. Now, we are going to recolor vertices $(v_{i+3},v_{i+2})$ and $v_{i+2}$ and color the $t$-vertex $v_{i+1}$ as well.  It can be easily seen that, there are two available colors for the $i$-vertex $(v_{i+3},v_{i+2})$ and $\Delta-1\geq 3$ available colors for each of the $t$-vertices $v_{i+1}$ and $v_{i+2}$. By coloring these vertices with different colors, we have a $(\Delta+3)$-coloring for $G^{\frac{3}{3}}$ with the desired property of the theorem.
}\end{proof}
According to Theorems \ref{delta4girth6} and \ref{delta5girth4}, any outerplanar graph $G$ is of $vi$-class 1 when $G$ is (1) triangle-free and $\Delta(G)\geq5$ or (2) $\Delta(G)=4$ and $g(G)\geq6$.
\begin{problem}{\rm
		Charactrize all outerplanar graphs of $vi$-class one.
}\end{problem}
{\bf Acknowledgements.} This research was in part supported by a grant from IPM (No.1400050116).

\end{document}